\documentclass[11pt]{amsart}
\usepackage[utf8]{inputenc}
\usepackage[pdftex]{graphicx}
\usepackage{amssymb}
\usepackage{amsmath}
\usepackage{amsfonts}
\usepackage{amsthm}
\usepackage{bm}
\usepackage{mathrsfs}
\usepackage{appendix}
\usepackage{enumerate}
\usepackage[left=2.5cm, right=2.5cm, bottom=3cm, top=3cm]{geometry}
\usepackage[stretch=30,shrink=30]{microtype}
\usepackage{stmaryrd}
\usepackage[percent]{overpic}

\usepackage[normalem]{ulem} 
\usepackage{mathtools}
\usepackage{todonotes}
\usepackage{pgf,tikz}
\usepackage{pgfplots}
\usetikzlibrary{decorations.markings}
\usetikzlibrary{arrows}

\usepackage{xcolor}
\definecolor{antiquefuchsia}{rgb}{0.57, 0.36, 0.51}
\definecolor{azure}{rgb}{0.0, 0.5, 1.0}
\newcommand{\fr}{\hfill$\blacksquare$}

\usepackage[colorlinks = true, linkcolor = black, citecolor = antiquefuchsia]{hyperref}

\makeatother

\numberwithin{equation}{section}

\newtheorem{theorem}{Theorem}[section]
\newtheorem{lemma}[theorem]{Lemma}
\newtheorem{proposition}[theorem]{Proposition}
\newtheorem{corollary}[theorem]{Corollary}

\newtheorem{definition}[theorem]{Definition}
\theoremstyle{remark}
\newtheorem{remark}[theorem]{Remark}

\newcommand{\N}{\mathbb{N}}
\newcommand{\Z}{\mathbb{Z}}
\newcommand{\R}{\mathbb{R}}

\newcommand{\mres}{\mathbin{\vrule height 1.6ex depth 0pt width
0.13ex\vrule height 0.13ex depth 0pt width 1.3ex}}

\newcommand{\T}{{\mathbf T}}

\DeclareMathOperator{\dist}{dist}
\DeclareMathOperator{\Per}{Per}

\renewcommand{\d}{{\rm d}}
\newcommand{\HH}{{\mathcal H}}
\newcommand{\nchi}{{\raise.3ex\hbox{$\chi$}}}

\renewcommand{\phi}{\varphi}
\newcommand{\eps}{\varepsilon}

\def\Z{\ensuremath\mathbb{Z}}

\usepackage[capitalise,nameinlink]{cleveref}
\crefformat{equation}{#2(#1)#3}
\crefrangeformat{equation}{#3(#1)#4 to~#5(#2)#6}
\crefmultiformat{equation}{#2(#1)#3}{ and~#2(#1)#3}{, #2(#1)#3}{ and~#2(#1)#3}

\newcommand{\abs}[1]{\left\lvert #1 \right\rvert}
\newcommand{\Abs}[1]{\left\| #1 \right\|}
\newcommand{\ENCLOSE}[1]{\left\{ #1 \right\}}

\title[]{Anisotropic isoperimetric double tilings of the plane}

\author[]{Francesco Nobili} 
\address{Università di Napoli Federico II, Dipartimento di Matematica, Via Cintia, Monte S. Angelo, 80126 Napoli (NA), Italy}
\email{\url{francesco.nobili@unina.it}}
\author[]{Matteo Novaga} 
\address{Università di Pisa, Dipartimento di Matematica, Largo Bruno Pontecorvo 5,
56127 Pisa, Italy}
\email{\url{matteo.novaga@unipi.it}}
\author[]{Emanuele Paolini} 
\address{Università di Pisa, Dipartimento di Matematica, Largo Bruno Pontecorvo 5,
56127 Pisa, Italy}
\email{\url{emanuele.paolini@unipi.it}}

\date{\today. 
}  
 
\begin{document}
\begin{abstract}
We study periodic partitions of the plane into two distinct 
cells minimizing the anisotropic $\ell_1$-perimeter. For a rectangular 
lattice $G$, we compute explicitly
the $(G,\ell_1)$-isoperimetric profile and classify all minimizers.
When one cell has small area, the optimal tiling is 
generated by a square and a chipped rectangle, 
while in the remaining regime
the tiling is generated by two adjacent rectangles. 
Further minimizing over all possible planar 
lattices, we show that the $\ell_1$-isoperimetric profile 
is attained by the Pythagorean double tiling 
of two axis-aligned squares sharing a vertex. 
This configuration is unique
unless the two cells are assigned the same area.
Finally, we prove that the limiting (non periodic) partitions obtained by sending one volume to infinity are locally $\ell_1$-isoperimetric.
\end{abstract}

\maketitle
\setcounter{tocdepth}{1}
\tableofcontents

\section{Introduction} 
The study of configurations with minimal interface energy is a longstanding research program in the calculus of variations and geometric measure theory, with applications ranging from material science, crystallography, and physics. Depending on the formulation and on the concept of interface energy, the natural targets of these theoretical investigations are isoperimetric-type inequalities with geometric and physical flavour. An important goal is typically the understanding and characterization of minimal-energy configurations.

A relevant example from discrete geometry is the so-called Kelvin problem (\cite{Kelvin1887}), which consists of finding a partition of the Euclidean space $\R^d$, for $d\ge 2$, into cells of equal volume so that the interface area is minimal. Except for the planar case $d=2$, where the problem takes the name of the Honeycomb conjecture and was finally solved in \cite{Hales01} (see also \cite{Fejes1,Fejer2,Fejer3,Fejer4,MorganChristopherGreenleaf1988,Morgan1999}), little is currently known for $d>2$ (cf.\ \cite{WeairePhelan1996}).

Another source of examples from geometric measure theory is the characterization of minimal clusters (\cite{Almgren76}). The central objects are multi-isoperimetric inequalities and the characterization of minimal clusters, a problem going under the name of the $k$-bubble problem. A $k$-bubble in $\R^d$ with $k\le d+1$ is a cluster with $k$ distinguished chambers with prescribed volumes that is conjectured to have minimal interface area (\cite{SullivanMorgan96}). Currently, several advances have been obtained in the series of works \cite{FoisyAlfaroBrockHodges93,CicaleseLeonardiMaggi17,MorganWichiramala02,DorffLawlorSampsonWilson09,HassHutchingsSchlafly95,HassSchlafly00,HutchingsMorganRitore02,Reichardt08,Wichiramala,PaoliniTamagnini18,PaoliniTamagnini18bis,PaoliniTortorelli20,MilmanNeeman22,MilmanNeeman24}, depending on the dimension $d$, the number of chambers $k$, and the prescribed volumes. We refer to \cite{MorganBOOK} and the survey \cite{Milman_ICM} for more insights and references.

For our goals, we recall that these geometric problems have also been investigated with anisotropic surface energies, namely with a different notion of interface energy that is often motivated by the study of material science and crystallography (\cite{Wulff01,Taylor78,Herring51}). Given a norm $\phi$ in $\R^d$, we recall the concept of $\phi$-surface energy, also called anisotropic perimeter, which is defined for each Borel set $E\subset \R^d$ as
\begin{equation}
    \Per_{\varphi} (E)\coloneqq \int_{\partial^*E}\phi (\nu_E(x))\,\d \HH^{d-1}(x).
\label{eq:per intro}
\end{equation}
We refer to Section \ref{sec:prelim} for the notations and concepts appearing in the above formula, and to the books \cite{AmbrosioFuscoPallarabook,Maggi12_Book} for a detailed introduction. It is well known that, in this scenario, the regularity theory for perimeter or cluster minimizers is different from the standard one; see \cite{FranceschiPratelliStefani23_anisotropic} for a thorough investigation.

For instance, double bubbles have also been investigated in this anisotropic scenario for the so-called Manhattan norm, namely $\phi(x) = \|x\|_{\ell_1}= \sum_i |x_i|$. We recall the work of \cite{MorganChristopherGreenleaf1988}, who initiated the study of Wulff planar clusters, and also the recent advances \cite{FriedrichGornyStefanelli24_discrete,FriedrichGornyStefanelli24,FriedrichGornyStefanelli25} up to three dimensions, and novel geometric settings \cite{FranceschiStefani19}.

\medskip 
In this work, we are interested in the study of periodic partitions of the plane with two distinct repeating cells, minimizing the total $\ell_1$-anisotropic length. Periodic isoperimetric problems for the standard isotropic area have been commented on in \cite{Morgan1999} and also numerically tested in \cite{FortesTeixeira01,FortesGranerTeixeira02}. Furthermore, these theoretical investigations align with both the Kelvin problem and the study of minimal clusters in terms of methodology and geometric flavour. Finally, as we are going to see in the next set-up paragraph, our goals are motivated by a recent series of works \cite{CesaroniNovaga22,NovagaPaoliniStepanovTortorelli22,NovagaPaoliniStepanovTortorelli23,CesaroniNovaga23_1,CesaroniFragalaNovaga23} aiming at studying isoperimetric periodic partitions of the Euclidean space, and its recent generalization with unequal cells in \cite{NobiliNovaga24,NobiliNovagaPaolini26}.

\medskip 

\noindent\textbf{Theoretical set-up}.
We consider $d$-dimensional lattices $G$, namely groups of translations in the Euclidean space $\R^d$. The fundamental domain of a lattice $G$ is a Borel set $D\subset\R^d$ satisfying
\[
    \abs{\R^d \setminus \bigcup_{g\in G} (D+g)}=0\qquad
    \abs{(D+g)\cap D}=0,\quad  \forall g\in G\setminus\ENCLOSE{0}.
\]
Since the volume of the fundamental domain only depends on the group $G$, we will denote the volume of the lattice $G$ by $|G|\coloneqq |D|$ for any (hence all) fundamental domains $D$. In the planar case, we will write ${\rm Area}(G)=|G|$. Given a lattice $G$ with ${\rm Area}(G)>0$, there are vectors $\{g_1,\dots, g_d\} \subset G$ called generators which are independent vectors and such that, for every $g\in G$, we have $g=\sum_{k=1}^d c_k g_k$ with $c_k\in \Z$. A canonical choice of fundamental domain is the parallelepiped spanned by the $g_i$'s, namely $D=\{\sum_{k=1}^d c_k g_k \colon c_k\in [0,1]\}$, such that $|D|=\abs{\det(g_1,\dots,g_d)}$.

A tiling in $\R^d$ is a collection $\T$ of Borel sets such that $\abs{E\cap F}=0$ for all $E,F\in \T$, $E\neq F$, and
$\abs {\R^d\setminus\bigcup_{E\in \T} E}=0$. A periodic $N$-tiling, for $N\in\N$, is a tiling $\mathbf T$ such that there are $N$ Borel sets $E_1,\dots,E_N$ called generators and a lattice $G$ such that we have
\[
\mathbf T = \{ E_i + g : i=1,\dots,N, g \in G\}.
\]
We shall shortly say that $\mathbf T$ is an $N$-tiling relative to $G$ generated by $E_1,\dots,E_N$. Moreover, as customary, we shall consider tilings up to Borel representatives and say that $\mathbf T$ is generated by $\tilde E_1,\dots,\tilde E_N$ also if $|E_i \triangle \tilde E_i|=0$ for $i=1,\dots,N$. We call $\mathbf T$ non-degenerate provided $|E_i|>0$ for every $i=1,\dots,N$. Given a norm $\phi \colon \R^d\to[0,\infty)$, we define the $\phi$-perimeter of a tiling $\mathbf T$ as
\[
\Per_\phi(\mathbf T) \coloneqq \frac 12\sum_{i=1}^N\Per_\phi(E_i),
\]
where $\Per_\phi(E)$ is the anisotropic $\phi$-perimeter of a Borel set $E\subset \R^d$ (cf.\ \eqref{eq:per intro}). A periodic $N$-tiling
$\mathbf T = \{E_i+g\colon i=1,\dots,N, g\in G\}$
is called $\phi$-isoperimetric relative to a lattice $G$ provided that for any other $N$-tiling
$\mathbf T'=\{E'_i+g\colon i=1,\dots,N, g\in G\}$
with $\abs{E_i}=\abs{E_i'}$, $i=1,\dots,N$, one has
\begin{equation}
\Per_\varphi(\mathbf T) \le \Per_\varphi(\mathbf T').\label{eq:isoperimetric rel G}
\end{equation}
In this note, we are interested in classifying $\phi$-isoperimetric double tilings relative to a lattice $G$ of the plane ($N=2$ and $d=2$). This can be seen as the extension of the case of $N=1$ cells, in which case the isoperimetric problem is equivalent to that of isoperimetric fundamental domains, initiated by \cite{Choe89} and subsequently studied in the series of works \cite{MartelliNovagaPludaRiolo2017,CesaroniNovaga22,NovagaPaoliniStepanovTortorelli22,CesaroniNovaga23_1,CesaroniFragalaNovaga23} dealing with different types of local and non-local perimeter functionals.

\medskip

\noindent\textbf{Main results}. We next state our main result on anisotropic isoperimetric double tilings of the plane with the so-called Manhattan anisotropy given by
\[
    \phi(x,y) = \Abs{(x,y)}_{\ell_1}= \abs{x} + \abs{y}.
\]
We shall also consider rectangular lattices $G$ generated by the vectors $g_1=(a,0),g_2=(0,b)$ for $a,b>0$ so that ${\rm Area}(G)=ab$. By scaling, it will be enough to suppose that $ab=1$. A key notion for our goals is that of the $(G,\ell_1)$-isoperimetric profile function defined by
\[
  \mathcal I_{G,\ell_1}(x):= \inf \left\{\Per_{\ell_1}(\mathbf T) \colon \mathbf{T} = \{ E_i +g \colon i=1,2,\ g\in G\}\text{ with } |E_1|=x,\ |E_2|=1-x\right\},
\]
for all $x\in[0,1]$. Observe that double tilings $\mathbf{T}$ minimizing $\mathcal I_{G,\ell_1}(x)$, whose existence is reported for completeness in Theorem \ref{thm:existence} below, are precisely $\ell_1$-isoperimetric double tilings relative to $G$, according to \eqref{eq:isoperimetric rel G}.

Our first main result is the explicit characterization of the isoperimetric profile together with the classification of its minimizers.
\begin{theorem}\label{thm:main result}
Let $G$ be the rectangular lattice with ${\rm Area}(G)=1$. Let us denote by $(a,0),(0,b)$ the generators of $G$ so that $ab=1$, and let us assume that $a\ge b$. Then, it holds
\[ 
   \mathcal I_{G,\ell_1}(x) =
  \begin{cases}
    2\sqrt{x} + a + b &\text{if }x \le \frac{b^2}{4},\\
    a + 2b &\text{if }  \frac{b^2}{4} < x \le \frac 12,\\
    \mathcal I_{G,\ell_1}(1-x),&\text{otherwise.}
  \end{cases}
\]
In particular, $[0,1]\ni x \mapsto \left(\mathcal I_{G,\ell_1}(x) -\mathcal I_{G,\ell_1}(0)\right)^2$ is concave. Furthermore, it holds:
\begin{itemize}
    \item[{\rm o)}] if $x=0$, then $\mathcal I_{G,\ell_1}(0)$ is minimized by the $1$-tiling generated by $[0,a]\times[0,b]$;
    \item[{\rm i)}] if $0<x \le \frac{b^2}{4}$, then $\mathcal I_{G,\ell_1}(x)$ is minimized by all the double tilings generated by an axis-aligned square and a chipped rectangle (see Figure \ref{fig:small quadratino} and Figure \ref{fig:uniqueness case 1});
    \item[{\rm ii)}] if $\frac{b^2}{4} < x \le \frac 12$, then $\mathcal I_{G,\ell_1}(x)$ is minimized by the double tilings generated by two adjacent rectangles (see Figure \ref{fig:strip}).
\end{itemize}
Finally, the minimizers described above are the only possible minimizers, up to translations. 
\end{theorem}
\begin{figure}
\centering
\begin{overpic}[scale=0.5]{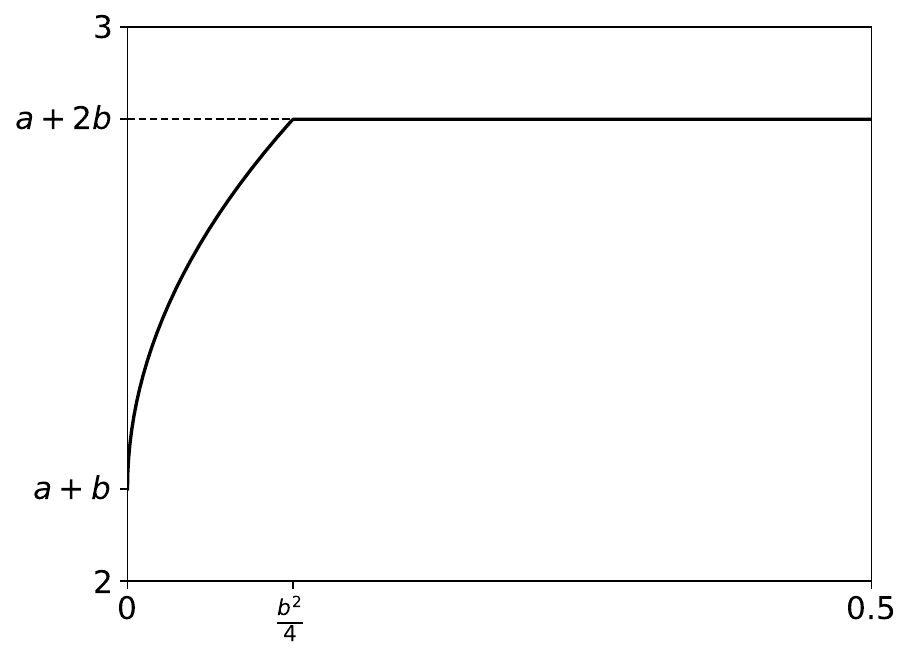}
    \put(19,25){\includegraphics[scale=0.3]{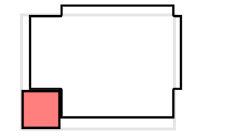}}
    \put(50,40){\includegraphics[scale=0.3]{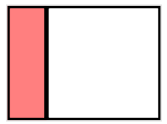}}
\end{overpic}
\caption{Isoperimetric profile $\mathcal I_{G,\ell_1}(x)$.}
\label{fig:profile}
\end{figure}
\begin{figure}
\includegraphics[scale=0.35]{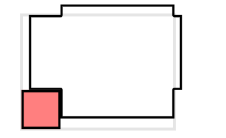}
\hspace{1cm}
\includegraphics[scale=0.18]{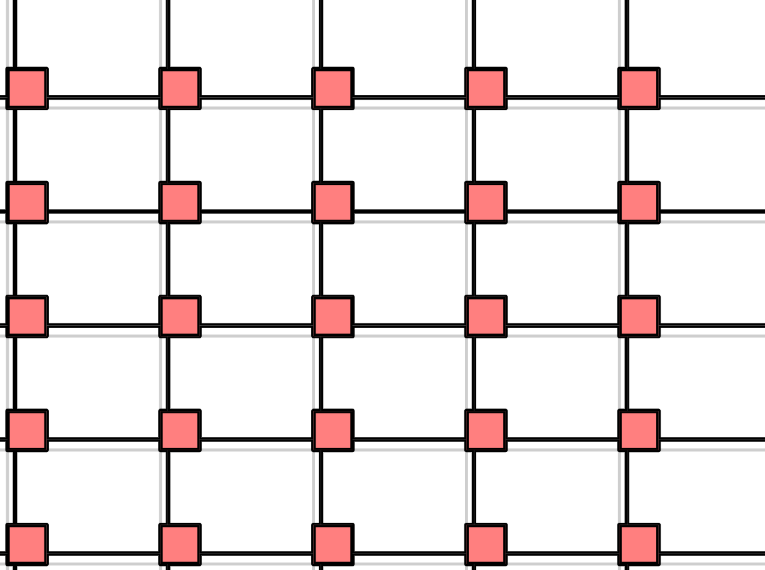}
\caption{The $\ell_1$-isoperimetric double tiling generated by a square and a chipped rectangle.}
\label{fig:small quadratino}
\end{figure}
\begin{figure}
\includegraphics[scale=0.35]{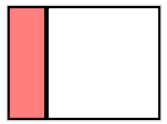}
\hspace{1cm}
\includegraphics[scale=0.18]{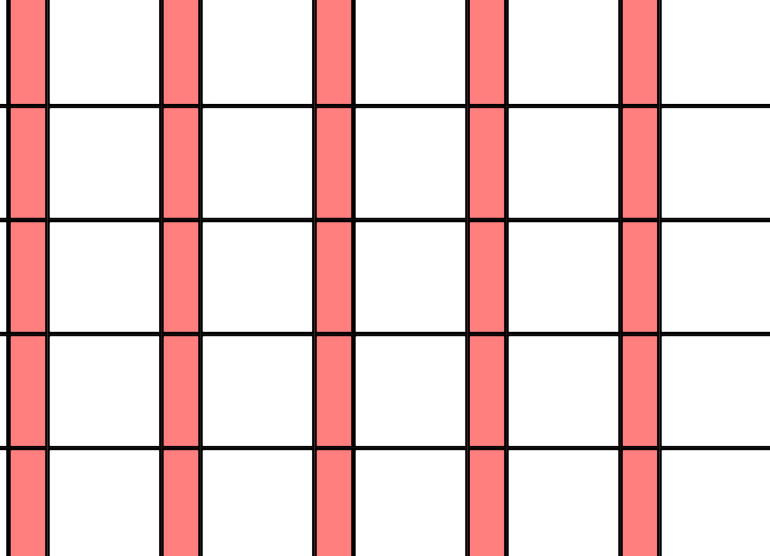}
\caption{The $\ell_1$-isoperimetric double tiling generated by two adjacent rectangles.}
\label{fig:strip}
\end{figure}
The above statement gives an explicit description of the $(G,\ell_1)$-isoperimetric profile function; see Figure \ref{fig:profile} for the plot of $x\mapsto \mathcal I_{G,\ell_1}(x) $ and its minimizers. First, we note that when the double tiling is degenerate, namely when $x\in \{0,1\}$, then the resulting $(G,\ell_1)$-isoperimetric $1$-tiling relative to $G$ is precisely the fundamental domain given by the parallelepiped spanned by the vectors $(a,0),(0,b)$. Instead, when $x \in (0,1)$, the minimizers are also completely classified, and the minimizing models are depicted in Figure \ref{fig:small quadratino} and Figure \ref{fig:strip}, attaining precisely the cost $\mathcal I_{G,\ell_1}(x)$ as $x\in(0,1)$. Finally, these models are unique up to translations, and we recall that in Case i) there are several energetically equivalent possibilities to allocate the small axis-aligned square as described in Figure \ref{fig:case 1 lenghts} (see some concrete examples in Figure \ref{fig:uniqueness case 1}). 
\begin{figure}
    \centering
    \includegraphics[scale=1.1]{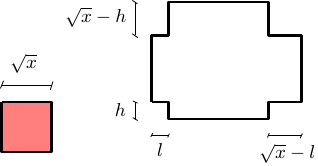}
    \caption{Given $x\le \frac{b^2}{4}$, any choice of $E_1$ (red) and $E_2$ (white) with $h,l \in [0,\sqrt{x}]$ generates a double tiling with perimeter equal to $\mathcal  I_{G,\ell_1}(x)$.}
    \label{fig:case 1 lenghts} 
\end{figure}
This uniqueness conclusion is rather different from that obtained in \cite[Theorem 3.5]{NobiliNovagaPaolini26} up to isometries. This is not surprising, as the $\ell_1$-perimeter is not invariant under arbitrary rotations and allows extra flexibility of minimal shapes, see e.g.\ \cite{MorganChristopherGreenleaf1988}.
\begin{figure}[ht]
    \centering
    \includegraphics[scale=.8]{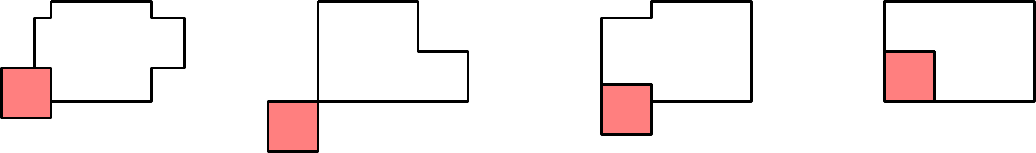}
    \caption{Example of generators with same total $\ell_1$-perimeter.}
    \label{fig:uniqueness case 1}
\end{figure}

\medskip

A natural isoperimetric profile function can be obtained by further minimizing among \emph{all} possible lattices $G$ of the plane (not just rectangular ones). In particular, one can consider the $\ell_1$-isoperimetric profile function defined by
\[
    [0,1]\ni x \mapsto \mathcal I_{\ell_1}(x)\coloneqq \inf \{\mathcal I_{G,\ell_1}(x)\colon G \text{ lattice with }{\rm Area}(G)=1\}.
\]
For the standard perimeter (i.e.\ for the choice of anisotropic norm given by $\phi(x) = \|x\|_{\ell_2}$), the above was considered by the authors in \cite{NobiliNovagaPaolini26}, and a surprisingly explicit description was achieved with three different minimizing models (cf.\ \cite[Theorem 3.5]{NobiliNovagaPaolini26}). In the current $\ell_1$-anisotropic setting, it turns out that the very same analysis is much more rigid.
\begin{theorem}\label{thm:pitagora}
    It holds
    \[
        \mathcal I_{\ell_1}(x) = 2\sqrt x + 2\sqrt{1-x},\qquad \forall x \in [0,1].
    \]
    Furthermore, for all $x\in[0,1]$, $x\neq \frac 1 2$, 
    the Pythagorean double tilings generated by two adjacent axis-aligned squares
    sharing one of their vertices (see Figure \ref{fig:pitagora}) are the only minimizers. 
    In particular, the squares tile the plane with 
    a lattice $G$ generated by two orthogonal unit vectors.

   Finally, if $x=\frac 1 2$ the only possible minimizers  are composed by alternating strips parallel to the coordinate axes (see Figure \ref{fig:meta area tiling}) filled by squares with the same side length $\frac 1 {\sqrt 2}$ that tile the plane with a lattice $G$ generated by the two vectors 
    $g_1 = (\frac 1 {\sqrt 2}, 0)$, 
    $g_2 = (h, \sqrt 2)$ or the symmetric ones 
    $g_1'=(0, \frac 1 {\sqrt 2})$, 
    $g_2' = (\sqrt 2, h)$ for any $h\in[0,\frac 1 {\sqrt 2}]$.
\end{theorem}
The above result is obtained by a direct application of the Wulff isoperimetric inequality for the $\ell_1$-perimeter. Indeed, see Theorem \ref{thm:Wulff} and \eqref{eq:isoperimetrica quadrata}, it always holds that $\mathcal I_{\ell_1}(x) \ge 2\sqrt x + 2\sqrt{1-x}$. The uniqueness part is instead obtained by a careful use of the rigidity of the Wulff shape in the $\ell_1$-isoperimetric inequality in the plane. We also refer to \cite{PGS91,MartiniMakaiSoltan98} for more insights around planar tilings generated by squares.
\begin{figure}
\includegraphics[scale=0.3]{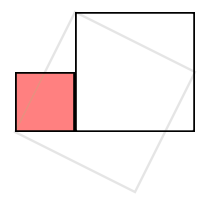}
\hspace{1cm}
\includegraphics[scale=0.2]{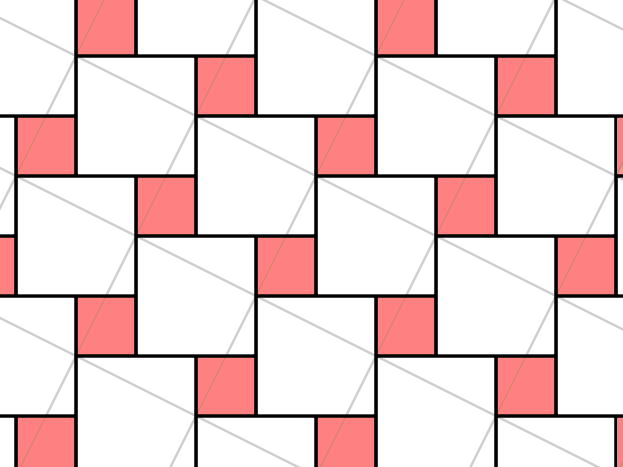}
\caption{The Pythagorean double tiling.}
\label{fig:pitagora}
\end{figure}

\medskip

The main results obtained in Theorem \ref{thm:main result} and Theorem \ref{thm:pitagora} are the starting point to further investigate the local minimality of the partitions induced by the $\ell_1$-isoperimetric double tilings. The natural concept, see Definition \ref{def:locally anisotropic isop}, is that of a locally $\phi$-isoperimetric partition.
\begin{remark}
    \rm
    We remark that in the case when $G$ is different from the lattice of the Pythagorean tiling and $x\neq \frac 1 2$, the isoperimetric tilings relative 
    to a rectangular lattice $G$ cannot induce a locally $\ell_1$-isoperimetric partition in the sense of Definition \ref{def:locally anisotropic isop}. In fact, denote by $(E_n)$ the partition of $\R^2$ induced by any of the $\ell_1$-isoperimetric double tilings relative to a rectangular lattice $G$, given by Theorem \ref{thm:main result}. Given a square $Q_r = [-r,r]^2$ with $r>0$, consider the partition $(F_n^r)$ 
    with $\abs{F_n^r}=\abs{E_n}$
    so that $E_n\triangle F_n^r \Subset Q_{r+1}$
    and $F_n^r$ is a square for all $n\in\N$ such that $E_n\subset Q_r$
    and fill the annulus $Q_{r+1} \setminus Q_r$ in a suitable way.
    Then, we have 
    \[
        \lim_{r\to\infty}\frac{\frac 12 \sum_{n\in \N}\Per_{\ell_1}(F^r_n,Q_{r+1}) }{4 r^2} = 2\sqrt{x}+2\sqrt{1-x} =\mathcal I_{\ell_1}(x).
    \]
    Since, by assumption, we have 
    $\mathcal I_{\ell_1}(x) < \mathcal I_{\ell_1,G}(x)$
    we deduce that $(E_n)$ cannot be a locally $\ell_1$-isoperimetric partition.

    On the other hand, the Pythagorean double tiling induces a locally $\ell_1$-isoperimetric partition, because every tile must satisfy 
    the isoperimetric inequality.\fr 
\end{remark}
In light of the above remark, a possible variant is to sent one area to infinity ($|E_2|\uparrow \infty$, hence also ${\rm Area}(G)\uparrow\infty$) while the remaining one is kept fixed ($|E_1|=1$). 
\begin{figure}
\includegraphics[scale=.75]{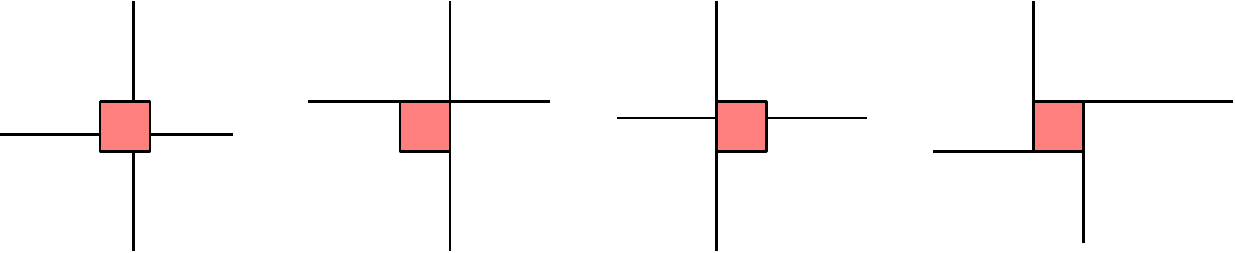}
\caption{Possible partitions obtained by sending $|E_2|\uparrow \infty$ while keeping $|E_1|=1$ on the isoperimetric double tilings: the first three partitions are obtained from Case i) in Theorem \ref{thm:main result}, the last one (the \emph{vortex}) is obtained from Theorem \ref{thm:pitagora}.}
\label{fig:global}
\end{figure}
\begin{theorem}\label{thm:locally isop}
Let $h,l \in [-1,1]$, 
$E_1 = (h,l)+[-1,1]^2 \subset \R^2$, 
and let
$E_2,E_3,E_4$ be the four connected components 
of $\R^2 \setminus (E_1 \cup \{x=0\} \cup \{y=0 \})$. 
Then the partition of the plane given by $\mathbf E = (E_1,\dots,E_5)$
is locally $\ell_1$-isoperimetric (cf.\ Definition \ref{def:locally anisotropic isop}).

The same property holds true if we consider 
the \emph{vortex} partition $\mathbf E$ given by $E_1=[0,1]^2$, and the connected components $E_2,E_3,E_4,E_5$  of 
$\R^2 \setminus(E_1 \cup \{y=1, x\ge 1\} \cup \{y=0,x\le 0\} \cup \{x=0,y\ge 1\} \cup \{x=1, y\le 0\})$
(see Figure~\ref{fig:global}).
\end{theorem}
The result above proves that certain $5$-partitions of the plane with four regions with infinite area are local minimizers of the anisotropic perimeter. This is indeed obtained by exploiting Theorem~\ref{thm:main result} and Theorem~\ref{thm:pitagora}, and by studying their local minimality properties. Finally, we recall the recent series of works studying minimality properties of clusters with chambers of possibly infinite assigned volume, \cite{AlamaBronsardVriend25,NovagaPaoliniTortorelli25,BronsardNeumayerNovackSkorobogatova25,NovagaPaoliniTortorelli25_nonunique,NeumayerNovackSkorobogatova26} for the isotropic perimeter, and \cite{Benitez26} in the anisotropic case.
\begin{figure}
\includegraphics[scale=0.25]{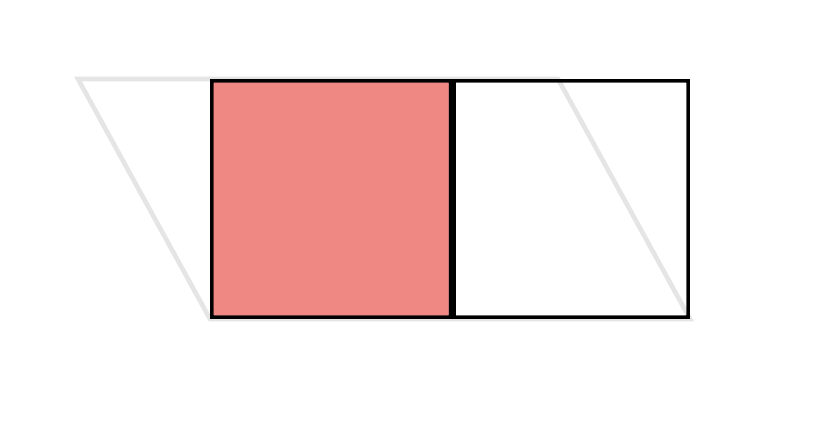}
\hspace{1cm}
\includegraphics[scale=0.2]{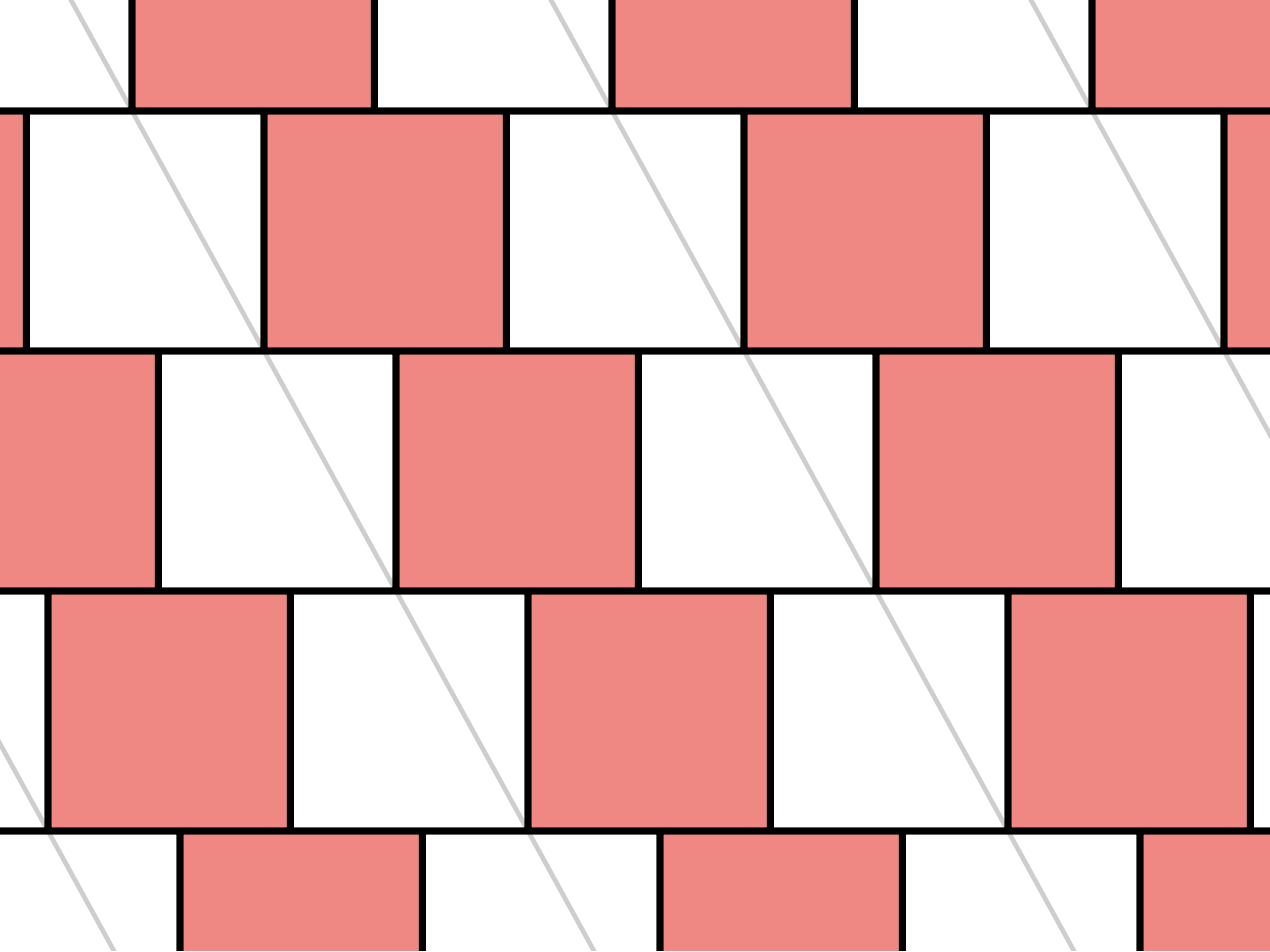}
\caption{A possible $\ell_1$-isoperimetric double tiling composed of two cells of equal area and with the same perimeter of the Pythagorean double tiling.}
\label{fig:meta area tiling}
\end{figure}
\section{Preliminaries}\label{sec:prelim} 
In this section, we recall some useful facts around the classical and anisotropic perimeters.
\subsection{Sets of finite perimeter}
We shall list some useful facts around sets of finite perimeter, referring to \cite{AmbrosioFuscoPallarabook} for a detailed introduction. For every $E\subseteq\R^d$ Borel and $\Omega\subseteq \R^d$, we recall that $E$ is said a set of finite perimeter in $\Omega$ if
\[
    \Per(E,\Omega)\coloneqq \sup\left\{ \int_E {\rm div}(X)\,\d \mathcal L^d \colon X \in C^\infty_c(\Omega,\R^d),\|X\|_{L^\infty}\le 1\right\}<+ \infty.
\]
The value $\Per(E,\Omega)$ is called the relative perimeter of $E$ in $\Omega$, and it extends to a locally finite nonnegative measure $\Per(E,\cdot)$ on the whole Borel $\sigma$-algebra. This will be called perimeter measure of $E$ and, when $\Omega=\R^d$, we shall simply write $\Per(E)\coloneqq \Per(E,\R^d)$. We then say that $E$ is a set of locally finite perimeter, provided $\Per(E,\Omega)<\infty$ is finite for all $\Omega\subset \R^d$ bounded open sets. In this case, there is a vector valued measure, denoted $D\nchi_E$, that satisfies $|D\nchi_E| = \Per(E,\cdot)$ and 
\[
    D\nchi_E = \nu_E \HH^{d-1}\mres{\partial^*E},
\]
where $\HH^{d-1}$ is the $(d-1)$-dimensional Hausdorff measure, and $\partial^*E$ is the reduced boundary of $E$, defined by $\partial^*E\coloneqq \{x \in \R^d \colon m(E,x)>0,m(E^c,x)>0 \}$, with $m(E,x) \coloneqq \lim_{r\to 0} \frac{\mathcal L^d(B_r(x)\cap E)}{\mathcal L^d(B_r(x)}$, and $\nu_E(x)$ is the outer unit normal of the set $E$, which is $\HH^{d-1}$-a.e.\ defined on $\partial^*E$. 
\subsection{Anisotropic surface energy}
In this note, we shall consider anisotropic variants of the standard Euclidean perimeter. We recall first some basic facts around anisotropic norms and the Wulff problem, referring to \cite[Chapter 20]{Maggi12_Book} for a detailed discussion.

Given a norm $\phi \colon \R^d\to [0,\infty)$, we define the anisotropic surface $\phi$-energy of a set $E\subset \R^d$ relative to an open set $\Omega\subset \R^d$ by
\[
    \Per_\varphi(E,\Omega)\coloneqq \int_{\partial^*E \cap \Omega}\phi(\nu_E)\,\d \HH^{d-1}(x) \in [0,\infty],
\]
where the above is canonically set to $+\infty$ if $E$ is not a set of locally finite perimeter. As before, we simply write $\Per_\phi(E)\coloneqq \Per_\varphi(E,\R^d)$. The Wulff problem is the following variant of the isoperimetric problem
\begin{equation}
    \inf \{ \Per_\varphi(E) \colon E\subset \R^d, \abs{E}=m\}, \qquad\forall m>0.\label{eq:Wulff}
\end{equation}
and the minimizers are called Wulff shape and are completely characterized by the next theorem (see, e.g.\ \cite[Theorem 20.8 and Eq. (20.14)]{Maggi12_Book}). Below, $\phi^*$ is the dual norm defined by $\phi^*(v) \coloneqq \sup\{  v\cdot x \colon \phi(x)<1\}$ for all $v \in \R^d.$
\begin{theorem}\label{thm:Wulff}
    For all $m>0$ there is $r>0$ such that $m=r^d|W_\phi|$, where $W_\phi \subset \R^d$ is the Wulff shape given by
    \[
        W_\phi \coloneqq \{ v \in \R^d \colon \phi^*(v)<1\}.
    \]
    Then, for every $x_0 \in \R^d$, the set $x_0+rW_\phi$ is a minimizer of \eqref{eq:Wulff}. In particular, for every $E\subset \R^d$ Borel, the Wulff inequality holds
    \[
        \Per_\phi(E) \ge n |W_\phi|^{\frac 1d}\abs{E}^{\frac{d-1}{d}}.
    \]
\end{theorem}
We also note, see \cite{BrothersMorgan94,FigalliMaggiPratelli10}, that the Wulff shape is also the unique minimizer up to translation (see also \cite{FigalliMaggiPratelli10,FuscoEspositoTrombetti05,Neumayer16} for quantitative improvements). The above result, in the planar case with the Manhattan anisotropy $\phi(x)=\|x\|_{\ell_1},\phi^*(v)=\|v\|_{\ell_\infty}$ reads as follows. Since $W_{\ell_1}\subset \R^2$ is a square with area $|W_\phi|=4$, Theorem \ref{thm:Wulff} gives 
\begin{equation}
    \Per_{\ell_1}(E)\ge 4\sqrt{\abs{E}},\qquad\forall E\subset \R^2, \label{eq:isoperimetrica quadrata}
\end{equation}
and equality occurs if and only if $E$ is equivalent to a square of edge-length $\sqrt{\abs{E}}$ and aligned with the coordinate axes, up to translations.
\section{Existence and regularity}
In this section, we shall discuss basic existence and topological regularity properties of anisotropic isoperimetric tilings, with special attention to the planar setting.

\subsection{Existence of anisotropic isoperimetric tilings}
The following existence result was obtained in \cite{NobiliNovaga24}, but a short proof is included for completeness with emphasis on a certain construction in the planar setting. We recall that an $N$- tiling $\mathbf T=\{E_i+g \colon i=1,\dots,N,\, g\in G\}$ is said locally finite if for all $i=1,\dots,N$ and for all compact set $C\subset \R^d$ it holds
\begin{equation}
    N_i \coloneqq \# G_i <\infty, \qquad \text{where }G_i\coloneqq \{ g \in G\colon |(E_i+g)\cap C|>0\}\label{eq:def Ni}
\end{equation}
Let us also define the packing and covering radius of $G$, respectively given by
\begin{align*}
\rho_G&\coloneqq\sup \{ r>0 \colon B_r(g_1)\cap B_r(g_2) = \emptyset \text{ for all }g_1\neq g_2 \in G\};\\
r_G&\coloneqq \inf\{r>0 \colon \cup_g (B_r(0)+g) = \R^d \}.
\end{align*}
\begin{theorem}\label{thm:existence}
    Let $N\in\N, d\ge 2$ and let $G$ be a lattice in $\R^d$. Let $v = (v_1,\dots,v_N)$ be so that $\sum_{i=1}^N v_i = |G|$. Let $\phi \colon \R^d \to [0,\infty)$ be a norm. Then, there exists a minimizer of 
    \[
        \inf \{ \Per_\phi(\mathbf T) \colon \mathbf{T}=\{E_i+g \colon i=1,\dots,N,\, g\in G\} \text{ with }|E_i|=v_i\}.
    \]
    In particular, the minimizer is a $\phi$-isoperimetric $N$-Tiling relative to $G$. Finally, if $d=2$, there is a minimizing tiling such that each generator is essentially bounded (in particular, it is locally finite).
\end{theorem}
\begin{proof}
    In \cite[Theorem 3.2]{NobiliNovaga24} (with the choice $\mu=0$), it was shown that there exists a minimizer of the following variational problem:
    \begin{equation}
        \inf_D \inf \left\{ \frac12\sum_{i=1}^N\Per(E_i) \colon \begin{array}{ll}
        E_i\subseteq D\text{ Borel}, &|E_i|=v_i,\quad i=1,...,N\\
         |E_i\cap E_l|=0, &\big|D\setminus \cup_i E_i \big|=0,
    \end{array}  \right\},\label{eq:variational existence}
   \end{equation}
   where the first infimum is taken among all fundamental domains $D$ of $G$, and the second infimum is taken 
   among all Borel partitions $(E_i)$ of $D$ with assigned areas $v_1,..,v_N$. 
   Evidently, $D,(E_i)$ minimizes the above if and only if 
   $\T \coloneqq \{ E_i + g \colon i=1,\dots,N, g \in G\}$ is an $\ell_1$-isoperimetric 
   $N$-tiling relative to $G$ with prescribed volumes $v_1,\dots,v_N$.

   Finally, if $d=2$, the last conclusion is implied by \cite[Proposition 5.1]{NobiliNovaga24}. Since we shall later need the precise construction, we report it briefly here. Let us consider the fundamental domain $D\coloneqq \cup_i E_i$ and its decomposition into indecomposable sets $(D_l)$ satisfying $|D|=\sum_l |D_l|, \Per_{\phi}(D)=\sum_ l\Per_\phi(D_l)$ (\cite{AmbrosioCasellesMasnouJean01}). For all $l\in\N$,  consider $g_l \in G$ so that $(D_l-g_l)\cap B_{r_G}(0) \neq \emptyset$ (here $r_G>0$ is the covering radius of the lattice $G$) and define the sets
    \begin{equation}
        E_i' \coloneqq \cup_l \big((E_i \cap D_l)-g_l\big),\qquad \tilde D_l \coloneqq D_l -g_l, \quad \forall i,l.
    \label{eq:decomposition}
    \end{equation}
    By construction, $ D' \coloneqq \cup_l \tilde D_l$ is also a fundamental domain for $G$ and $( E_i')$ is a partition up to negligible sets of $ D'$. By construction, it also holds
    \[
        \sum_i\Per_\varphi( E_i') =  \sum_i\Per_\varphi( E_i).
    \]
    and in particular $\mathbf T' \coloneqq \{E_i'+g \colon i=1,\dots,N,\, g\in G\}$ is a $\phi$-isoperimetric $N$-tiling relative to $G$ satisfying $| E_i'|=v_i$ for all $i=1,\dots,N$. Moreover, we have
    \[
        {\rm diam}( E_i')\le {\rm diam}(\tilde D) \le 2r_G + \sum_l {\rm diam}(D_l) \le  2r_G + c \sum_l \Per_\varphi(D_l)\le 2r_G + c' \Per_\varphi(\mathbf T)<\infty,
    \]
    for some dimensional constants $c,c'$, having used the diameter to perimeter estimates for indecomposable planar sets of finite perimeter (cf.\ \cite[Lemma 2.13]{DayrensMasnouNovagaPozzetta22}).
    Therefore, the partition $( E_i' + g)_{i,g\in G}$ is locally finite, and the generators $E_i'$ are essentially bounded, by construction.
\end{proof}
\subsection{Topological regularity of anisotropic tilings}
In this part, we establish basic density estimates and topological regularity properties of anisotropic tilings. 

We first start with a definition of local minimality for partitions, and then show that it is satisfied by an isoperimetric tiling. 
\begin{definition}\label{def:local minimal}
    Let $\phi$ be a norm in $\R^d$, let $(E_n)$ be a Borel partition of an open set $ \Omega\subset \R^d$. We say that $(E_n)$ is a volume constrained minimizer of the $\phi$-perimeter relative to $\Omega$, provided for every $B$ open bounded subset of $\Omega$ and every partition $(F_n)$ such that
    \[
        |E_n\cap \Omega|=|F_n\cap \Omega|,\qquad E_n \triangle F_n \Subset B,\qquad \forall n\in \N,
    \]
    it holds
    \[
        \sum_{n\in\N}\Per_\phi(E_n,B)\le \sum_{n\in\N}\Per_\phi(F_n,B).
    \]
\end{definition}
\begin{lemma}\label{lem:local minimality}
     Let $\phi$ be a norm in $\R^d$, let $G$ be a lattice, and let $\mathbf T\coloneqq \{E_i+ g \colon g \in G,i=1,\dots,N \}$ be a $\phi$-isoperimetric $N$-tiling relative to a lattice $G$. Let $ \Omega \subset \R^d$ be an open set such that $\overline \Omega \cap (\overline \Omega + g) =\emptyset$ for all $g \in G\setminus \{ 0 \}$ (for instance, $\Omega=B_\rho(x)$ for all $x\in\R^d,\rho<\rho_G$). Then, the partition $(E_i +g)_{i,g}$ of $\R^d$ is a volume constrained  local minimizer of the $\phi$-perimeter relative to $\Omega$.
\end{lemma}
\begin{proof}
    Consider a fundamental domain $D\supset\Omega$ with the property $\Per_\phi (E_i,\partial D) = 0$ for all $i=1,\dots,N$. This is possible by assumptions on $\Omega$, and up to an arbitrary small translation of $D$. Let $(g_l)$ for $l \in\N$ be an enumeration of $G$ and let us denote for brevity $E_{i,l} \coloneqq E_i + g_l$ for every $l \in \N,i$.

    Let $(F_{i,l})$ be another partition of $\R^d$ with the property that $E_{i,l}\triangle F_{i,l}\Subset B$ and $|E_{i,l}\cap \Omega|=|F_{i,l}\cap \Omega|$ for all $i,l$. Set
    \[
        F_i \coloneqq  \left( E_i \setminus \bigcup_{ l} (\Omega-g_l)\right)\cup \left(\bigcup_{l} (F_{i,l} -g_l)\right).
    \]
    We note that $\tilde D \coloneqq \cup_i F_i$ is a fundamental domain for $G$, using that $D$ is a fundamental domain and that $(D +g_l) \triangle (\cup_iF_{i,l}) \Subset \Omega$ by assumptions. Define then the $N$-tiling
    \[
        \mathbf T' \coloneqq \{ F_{i}+g \colon g \in G,i=1,\dots,N\},
    \]
    that is periodic relative to $G$ and satisfies by construction $|F_i|=|E_i|$ for all $i=1,\dots,N$. Using at first that $\mathbf T$ is $\phi$-isoperimetric relative to $G$, we deduce
    \begin{align*}
         0 & \le \Per_\phi(\mathbf T')-\Per_\phi(\mathbf T) = \frac 12 \sum_i\Per(F_i) -\frac 12\sum_i\Per(E_i)  \\
         &\le\sum_i\Per_\phi\big(F_i,\cup_l (\Omega-g_l)\big) -\sum_i\Per_\phi\big(E_i, \cup_l (\Omega-g_l)\big) \\
         &\le  \sum_{i,l}\Per_\phi(F_{i,l},\Omega)-  \sum_{i,l}\Per_\phi(E_{i,l}, \Omega)  
    \end{align*}
    having used the subadditivity of the perimeter on disjoint sets, its translation invariance, and the choice of $D$. This concludes the proof.
\end{proof}
Regularity theory is by now well understood in the case of sets and clusters, for the classical isotropic perimeter. Many basic regularity results, however, also hold true for the $\phi$-anisotropic perimeter for any norm $\phi$ (without any regularity assumptions). We refer to the general references \cite{Almgren76,MorganChristopherGreenleaf1988,Maggi12_Book,FranceschiPratelliStefani23_anisotropic} for further details on the regularity of isoperimetric clusters. The following result establishes a basic topological regularity result for partitions minimizing the anisotropic perimeter, with an omitted proof as it would not bring any novelty. The idea is the same in all cases: first, a volume fixing result must be established (Almgren's lemma for partitions), this will trade the volume fixing constraint with a quasi-minimality property (see \cite[Appendix A]{BonaciniCristoferiTopaloglu2025} for a similar setting), then the co-area formula will produce a differential inequality to obtain a suitable decay of the densities on small balls.
The constants might depend on the partition, and on a prescribed distance from the boundary of the open domain, but are uniform on any such point.
\begin{theorem}[density estimates for partitions]\label{th:almgren}
Let $\phi\colon \R^d\to [0,+\infty)$ be a norm. 
Let $\Omega\subset \R^d$ be an open and connected set.
Let $(E_1,\dots,E_M)$ be a volume constrained minimizer of the $\phi$-perimeter relative to $\Omega$.
Then, for every $\eps>0$ there exists $c>0$, $\rho_0>0$ such that if $x\in \Omega$, $\dist(x,\Omega^c)>\eps$ it holds
\begin{equation}
  \abs{E_n \cap B_\rho(x)} \ge c \rho^d, \qquad
  \Per_\phi(E_n, B_\rho(x)) \ge c \rho^{d-1},
\end{equation}
for all $\rho<\rho_0$ and $n=1,\dots,M.$
\end{theorem}
Leveraging on the density estimates for partitions minimizing the $\phi$-perimeter, we will establish a crucial regularity property for anisotropic isoperimetric tilings that are locally finite (cf.\ \eqref{eq:def Ni}). 
\begin{theorem}[regularity of isoperimetric tilings]\label{th:regularity}
Let $N\in\N$, let $\phi$ be a norm in $\R^d$, let $G$ be a lattice, and let $\mathbf T\coloneqq \{E_i+ g \colon g \in G,i=1,\dots,N \}$ be a $\phi$-isoperimetric $N$-tiling relative to the lattice $G$. Assume that $(E_i+g)_{i,g}$ is a locally finite partition. Then there are $c,\rho_0>0$ such that for all $x\in \R^d$ and $i=1,\dots, N$, it holds
\[
    \abs{E_i\cap B_\rho(x)} > c \rho^d, \qquad
    \Per_\phi(E_i,B_\rho(x)) > c \rho^{d-1},\qquad\forall\rho\le \rho_0.
\]
In particular, there exists a representative $\tilde E_i$ with $|E_i\triangle \tilde E_i|=0$ and $\HH^{d-1}(\partial \tilde E_i \setminus \partial^* \tilde E_i)=0$.
\end{theorem}
\begin{proof}
Let $D\subset \R^d$ be bounded fundamental domain for the lattice $G$, take $r< \rho_G$ and consider $x_1,\dots, x_n \in \R^d$ such that the family of balls $(B_{r/2}(x_k))$ is an open covering for the compact set $\bar D$.
For $k=1,\dots,n$, we apply Lemma \ref{lem:local minimality} and Theorem~\ref{th:almgren} (with $\eps=r/4$) with the partition obtained by restricting $(E_i+g)_{i,g\in G}$ to the ball $\Omega = B_{r}(x_k)$. Observe that, since the $N$-tiling $\mathbf T$ is assumed to induce a locally finite partition, this restriction yields each time a \emph{finite} partition of $\Omega$.

Consider the constants $c>0$, $\rho_0>0$ given by each application of Theorem~\ref{th:almgren} as $k=1,\dots,n$. By selecting the worst possible constants, we can assume that $c$ and $\rho_0$ do not depend on $k$. For any $x\in \R^d$, there is $g\in G$ such that $x+g \in D$ and there is $k$ such that 
$x+g \in B_{r/2}(x_k)$. In particular, for every $\rho < r/4$, it holds that $B_\rho(x+g) \subset B_{r}(x_k)$ and $\dist\left(x+g,\R^d \setminus B_{r}(x_k)\right)>\eps$. Hence the conclusion of Theorem~\ref{th:almgren} yields the desired
measure and perimeter density estimates for each $E_i$ in $B_\rho(x+g)$. 
The same estimates therefore hold at the point $x$, by $G$ translation invariance. Finally, the last conclusion holds as a byproduct of the density estimates.
\end{proof}
As a byproduct of the above, we can obtain two crucial results for planar anisotropic isoperimetric tilings. These will be essential to prove our main result.
\begin{corollary}\label{cor:planar regularity}
    Let $N\in\N$ and let $G$ be a lattice in the plane $\R^2$. Let $v = (v_1,\dots,v_N)$ be so that $\sum_{i=1}^N v_i = |G|$. Let $\phi \colon \R^d \to [0,\infty)$ be a norm. Then, there exists a $\phi$-isoperimetric $N$-tiling $\mathbf T\coloneqq \{E_i+ g \colon g \in G,i=1,\dots,N \}$ relative to a lattice $G$ satisfying $|E_i|=v_i$, and each generator $E_i$ admits an open and bounded representative $\tilde E_i$ with $|E_i\triangle \tilde E_i|=0$ and $\HH^1(\partial \tilde E_i \setminus \partial^* \tilde E_i)=0$.
\end{corollary}
\begin{proof}
    By the last conclusion of Theorem \ref{thm:existence}, we know that there is a minimizing tiling $\mathbf T\coloneqq \{E_i+ g \colon g \in G,i=1,\dots,N \}$ inducing a locally finite partition and with $E_i$ essentially bounded. Therefore, the regularity result in Theorem \ref{th:regularity} applies, giving in turn the existence of open and bounded representatives  $\tilde E_i$ satisfying the desired properties.
\end{proof}
\begin{lemma}\label{lm:connected}
Let $\phi$ be a norm in the plane $\R^2$ and let $\mathbf T\coloneqq \{E_i+ g \colon g \in G,i=1,2 \}$ be a $\phi$-isoperimetric double tiling relative to a lattice $G$. Assume that $E_1,E_2$ are open and bounded generators of $\mathbf T$ such that $\HH^1(\partial E_i\setminus\partial^*E_i)=0$, $i=1,2$. If $\R^d\setminus (\bar E_1+G)$
is connected, then $E_2$ is connected.
\end{lemma}
\begin{proof}
Suppose, by contradiction, that $E_2$ can be decomposed as the union 
of two open disjoint sets $E_2'$, $E_2''$.
Consider a point $x'\in E_2'$ and a point $x''\in E_2''$.
By the assumptions on $E_1$ we know that there exists 
a curve contained in $U=\R^d\setminus (\bar E_1+G)$ connecting $x'$ to $x''$. 
By continuity there is a point $x$, along the curve,
such that $x \in (\partial E_2' \cap \partial (E_2''+g)) \setminus \bar E_1$ for some $g\in G$.
By perturbing the curve, we can also suppose that $x\in (\partial^*E_2'\cap \partial^* (E_2''+g))\setminus \bar E_1$.
Consider $\rho>0$ such that $B_\rho(x)\cap \bar E_1=\emptyset$. 
By Theorem~\ref{th:regularity} we know that
$\HH^1(\partial E_2'\cap \partial (E_2''+g)\cap B_\rho(x))>0$.

We can hence define $F_1\coloneqq E_1,F_2 \coloneqq E_2'\cup (E_2''+g)$
and notice that $ \mathbf T'\coloneqq \{F_i+ g \colon g \in G,i=1,2 \}$ is another double tiling relative to $G$
with the same prescribed areas. However we have 
\[
\Per_\phi(F_2) \le \Per_\phi(E_2) - \HH^1(\partial^* E'_2\cap\partial^*(E_2''+g))
 < \Per_\phi(E_2),
\]
contradicting the minimality of $\mathbf T$. This concludes the proof.
\end{proof}

\section{Estimates of the $\ell_1$-perimeter of double tilings}
In this section, we develop key properties for the $\ell_1$-perimeter of double tilings.
\subsection{Preliminary lemmas}
Since, in this work, we are interested in planar sets, it will be convenient to decompose the $\ell_1$-perimeter of a set $E\subset \R^2$ of finite perimeter into a horizontal and vertical contribution. To this aim, we set
\[
    P_V(E) \coloneqq  \int_{\partial^*E} \abs{\nu_E \cdot e_1}\, d\HH^{1},\qquad P_H(E) \coloneqq  \int_{\partial^*E} \abs{\nu_E \cdot e_2}\, d\HH^{1},
\]
so that we have can rewrite $ \Per_{\ell_1}(E) = P_H(E) + P_V(E).$ A first important observation is the following simple lemma, where we denote by $\pi_1,\pi_2$ the projections respectively on the $x,y$ axes.
\begin{lemma}\label{lm:123333}
    Let $E\subset \R^2$ be a set of finite perimeter. Then we have
    \begin{equation}\label{eq:438576}
      P_H(E) \ge 2 \HH^1(\pi_1(\partial^* E)), \qquad 
      P_V(E) \ge 2 \HH^1(\pi_2(\partial^* E)).
    \end{equation}
    Finally, it also holds
    \[
        \abs{E} \le \frac14 P_H(E) \cdot P_V(E).
    \]
\end{lemma}
\begin{proof}
    We subdivide the proof into different steps.
    
    \noindent\textsc{Step 1}. Let us consider $(E_i)_{i\in I}$ the family of indecomposable components of $E$, namely the collection of Borel subsets $E_i\subset E$ satisfying
    \[
        \abs{E}=\sum_{i\in I}\abs{E_i},\qquad \Per_{\varphi}(E)= \sum_{i \in I}\Per_\varphi(E_i).
    \]
    See \cite{AmbrosioCasellesMasnouJean01} for the existence of such a decomposition. By the perimeter-diameter estimates for planar indecomposable sets of finite perimeter (see, e.g.\ \cite[Lemma 2.13]{DayrensMasnouNovagaPozzetta22}), we also deduce that $E_i$ is (essentially) bounded and, up to change of representative, we can assume that $E_i$'s are bounded sets. For all $i\in I$, we can define $\tilde E_i = {\rm sat}(E_i)$, where the saturation ${\rm sat}(E)$ of a set $E$ is the union of all the bounded indecomposable components of $E^c$ (i.e.\ the union of $E$ and its \emph{holes}). Since $E_i$ is indecomposable, then $\tilde E_i$ is simple. Hence, by \cite{AmbrosioCasellesMasnouJean01}, there exists a Jordan curve $\gamma_i$ such that $\partial^*\tilde E_i = \Gamma_i\coloneqq {\rm Im}(\gamma_i)$ (modulo $\HH^1$-a.e.\ equality).

    \noindent\textsc{Step 2}. We show the first in \eqref{eq:438576}, the second being analogous. By the previous discussion, we can estimate
    \begin{align*}
        P_H(E)=\sum_{i\in I}P_H(E_i)\ge \sum_{i\in I}P_H(\tilde E_i).
    \end{align*}
    Since $\gamma_i$ is a Lipschitz loop spanning $\partial\tilde E_i$ modulo $\HH^1$-a.e., we deduce
    \[
        P_H(\tilde E_i)\ge 2 \HH^1(\pi_1(\Gamma_i)),\qquad\forall i \in I.
    \]
    The sought estimate follows, by sub-additivity and noticing  
    \[
        \sum_{i\in I}\HH^1(\pi_1(\Gamma_i)) \ge \HH^1(\cup_{i\in I}\pi_1(\Gamma_i)) = \HH^1(\pi_1(\cup_{i\in I}\Gamma_i)) = \HH^1(\pi_1( \cup_i\partial^*\tilde E_i)) = \HH^1(\pi_1(\partial^* E)),
    \]
    having used, lastly, that the \emph{holes} in each $E_i$ do not contribute to the total length of the projection.

    \noindent\textsc{Step 3}. For all $i \in I$, we can then denote $R_i$ the smallest rectangle containing $E_i$ aligned with the coordinate axes, and we infer, by the previous step, that
    \[
        \abs{E_i}\le \abs{R_i},\qquad P_H(E_i)\ge P_H(R_i),\qquad P_V(E_i)\ge P_V(R_i).
    \]
    Indeed, the first is obvious while the claimed estimates on the perimeter are a direct consequence of \eqref{eq:438576}. Therefore, we conclude the proof observing
    \begin{align*}
        \abs{E}=\sum_{i\in I}\abs{E_i}&\le \sum_{i\in I}\abs{R_i} = \frac14 \sum_{i \in I}P_H(R_i)P_V(R_i) \le \frac14 \sum_{i \in I}P_H(E_i)P_V(E_i)\\
        &\le \frac14 \sum_{i \in I}P_H(E_i)\sum_{i \in I}P_V(E_i) =  \frac14 P_H(E)\cdot P_V(E).
    \end{align*}
\end{proof}
We next show a result aiming at characterizing rectangular sets. This will be needed to deduce the uniqueness part of our main results.
\begin{lemma}\label{lm:43io0568}
Let $h,b >0$ and suppose that $E\subset \R^2$ is an open set of finite perimeter with $\HH^1(\partial E\setminus \partial^*E)=0$, $\abs{E}=hb$, $P_H(E)=2h$, $P_V(E)=2b$.
If $\bar E\supset\{0\}\times [0,b]$ then a horizontal translation of $\bar E$ is $[0,h]\times [0,b]$.
\end{lemma}
\begin{proof}
Consider the Borel function $u(t):=\HH^1(E\cap \{y=t\})$.
By~\eqref{eq:438576}, we have 
\[
2u(t) 
    = 2\HH^1(\pi_1(E\cap \{y=t\})) 
    \le 2\HH^1(\pi_1(E))
    \le P_H(E)
    = 2h.
\]
On the other hand, we have
\[
 hb = \abs{E}
 = \int_{-\infty}^{+\infty} u(y)\, dy 
 = \int_0^b u(y)\, dy,
\]
where, the last equality follows from $2b = P_V(E) = 2\HH^1(\pi_2(E)) \ge 2\HH^1(\pi_2(\{0\}\times [0,b])) = 2b$, which guarantees that 
$u(y)=0$ when $y\notin [0,b]$.
Putting together we deduce $u(y)=h$ for a.e.\ $y\in [0,b]$.
Since $\HH^1(\pi_1(E))=h$ while $\HH^1(E\cap \{y=t\})=h$ for 
all $t\in[0,b]$ we conclude that $\bar E$ is actually a rectangle.
\end{proof}
\subsection{Estimates for the $\ell_1$-perimeter}
In this part, we produce effective estimates on $2$-tilings that are periodic with respect to rectangular lattices. We shall always assume that $\mathbf T$ is a tiling generated by open sets $E$ satisfying $\HH^1(\partial E_i \setminus \partial^* E_i)=0$. As the following lemmas will be invoked on the $\ell_1$-isoperimetric tilings given by Corollary \ref{cor:planar regularity}, this topological regularity will be satisfied.

We first show a useful lower bound on the vertical and the horizontal anisotropic perimeter of the generators.
\begin{lemma}\label{lm:32323}
Let $\mathbf T$ be a $2$-tiling relative to a rectangular lattice 
$G$ in $\R^2$ generated by two vectors $(a,0)$, $(0,b)$, $a,b>0$.
Let $E_1$, $E_2$ be two generators of $\mathbf T$ and suppose they are open set with $\HH^1(\partial E_i \setminus \partial^* E_i)=0$. Then, it holds
\[
 P_H(E_1) + P_H(E_2) \ge 2 a, \qquad 
 P_V(E_1) + P_V(E_2) \ge 2 b.
\]
\end{lemma}
\begin{proof}
Since a line cannot be contained in the interior of $(E_1\cup E_2) + G$ (the union of the translations of the interiors of the tiles), every line must 
touch $\partial \mathbf T \coloneqq \cup_{i=1}^N \partial^*E_i$. This means that given any vertical line $v$ there exists $g\in G$ such
that $(v+g) \cap (\partial E_1 \cup \partial E_2) \neq \emptyset$.
By taking the lines $\{x=t\}$ with $t\in [0,a]$ we notice that the projection 
of $(\partial E_1 \cup \partial E_2) + a \mathbb Z$ on the $x$-axis, covers the whole segment $[0,a]$.
Hence, by~\eqref{eq:438576}, we have
\[
P_H(E_1) + P_H(E_2) \ge 2 (\HH^1(\pi_1(\partial E_1))+ \HH^1(\pi_2(\partial E_2)))
\ge 2 a.
\]
The second conclusion can be obtained similarly.
\end{proof}
Next, we show a lower bound on the $\ell_1$-perimeter of a double tiling when one generator is touched by every horizontal and vertical line. 
\begin{proposition}
\label{prop:326576}
Let $G$ be the rectangular lattice in $\R^2$ generated by two vectors $(a,0)$, $(0,b)$, $a,b>0$. 
Suppose that $\mathbf T$ is a 2-tiling relative to $G$
generated by two open sets $E_1,E_2$ with $\HH^1(\partial E_i\setminus \partial^*E_i)=0$ for $i=1,2$. Suppose that every horizontal line and every vertical line touches $\bar E_2 + G$.
Then
\begin{equation}
\label{eq:cost32}
\Per_{\ell_1}(\mathbf T) \ge a + b + 2 \sqrt{\abs{E_1}}.
\end{equation}

Moreover, if equality holds in~\eqref{eq:cost32} then
$\sqrt{\abs{E_1}} < \min\{a,b\}$ and
$\mathrm T$ is a double tiling generated by a square of area $\abs{E_1}$ and chipped rectangle obtained from a translation of $[0,a]\times [0,b]$.
\end{proposition}
\begin{proof}
By the assumptions on $E_2$, we have $\HH^1(\pi_1(\partial^* E_2)) = a$, $\HH^1(\pi_2(\partial^* E_2))=b$. Therefore, 
\eqref{eq:438576} yields
\begin{equation}\label{eq:aabb}
\begin{aligned}
  \int_{\partial^*E} \abs{\nu_E \cdot e_2} \,\d\HH^1\ge  a, \qquad
  \int_{\partial^*E} \abs{\nu_E \cdot e_1} \,\d\HH^1\ge  b,
\end{aligned}
\end{equation}
whence
\[
  \Per_{\ell_1}(E_2) \ge 2a + 2b.
\]
On the other hand, by the $\ell_1$-isoperimetric inequality \eqref{eq:isoperimetrica quadrata}, we have 
$\Per_\phi(E_1)\ge 4\sqrt{\abs{E_1}}$, with equality holding if 
and only if $E_1$ is equivalent to the Wulff shape, namely a square with edges of length $\sqrt{\abs{E_1}}$ and parallel to the coordinate axes. The first conclusion is therefore proven.

Suppose now that equality holds in~\eqref{eq:cost32}.
Let $I=\pi_1(E_2)$, $J=\pi_2(E_2)$
so that $\HH^1(I) = a$, $\HH^1(J)=b$
while $\Per_\phi(E_2)=2a+2b$.
Since $E_2$ is connected 
(by Lemma~\ref{lm:connected}) 
both $I$ and $J$ must be connected too.
Hence $E_2$ is contained in $F := (I\times J)\setminus (E_1+G)$ which is a set of 
area $ab-\abs{E_1}$. Since $ab-\abs{E_1} = \abs{E_2}$
we conclude that $E_2 = F$ (up to a negligible set).

Without loss of generality suppose now that $E_1=(0,\ell)^2$ and let $I=(\alpha,\alpha+a)$, $J=(\beta,\beta+b)$, with $0\le \alpha <a$ and $0\le \beta < b$, $\ell:=\sqrt{\abs{E_1}}$.
We now claim that $\alpha\le\ell$ and $\beta\le \ell$.
In fact if $\alpha>\ell$ we would have $\Per_\phi(E_2)\ge 2a + 2b + 3 \ell$ because at least 
three sides of the square $E_1$ would be \emph{inside} $E_2$. An analogous argument shows that $\beta>\ell$ cannot occur either, concluding thus the proof.
\end{proof}
Finally, we show a last bound on the $\ell_1$-perimeter of a double tiling in the situation in which the two generators are touched by two (possibly different) vertical lines.
\begin{lemma}
\label{prop:39899}
Let $G$ be the rectangular lattice in $\R^2$ generated by two vectors $(a,0)$, $(0,b)$, $a,b>0$. 
Suppose that $\mathbf T$ is an isoperimetric 2-tiling relative to $G$
generated by two open sets $E_1,E_2$ with $\HH^1(\partial E_i\setminus \partial^*E_i)=0$ for $i=1,2$.
Suppose there exists a vertical line contained in $\bar E_1+G$
and another vertical line contained in $\bar E_2 + G$.
Then
\begin{equation}
    \Per_{\ell_1}(\mathbf T) \ge a + 2b.
\label{eq:cost rectangles}
\end{equation}
Moreover, if equality holds in \eqref{eq:cost rectangles}, then the double 
tiling $\mathbf T$ is generated by two adjacent rectangles of area respectively $\abs{E_1},ab-\abs{E_2}$ whose union has closure equal to $[0,a]\times[0,b]$.
\end{lemma}
\begin{proof}
By Lemma~\ref{lm:32323} we have 
\[
P_H(E_1)+P_H(E_2) \ge 2a.
\]
On the other hand, since every horizontal line crosses both vertical lines and hence 
crosses both sets $E_1$ and $E_2$, we have, using~\eqref{eq:438576},
\[
P_V(E_1)\ge 2b, \qquad P_V(E_2)\ge 2b.
\]
Hence 
\[
\Per_\phi(\mathbf T) = \frac{P_H(E_1)+P_H(E_2)+P_V(E_1)+P_V(E_2)}{2} \ge a + 2b.
\]

Suppose now that equality holds. This means that 
\[
P_V(E_1) = P_V(E_2) = 2b
\]
and 
\[
P_H(E_1) = 2h, \qquad P_H(E_2) = 2(a-h)
\]
for some $h\in [0,a]$. We also note, by Lemma \ref{lm:123333}, that
\[
    ab = \abs{E}= \abs{E_1}+ \abs{E_2} \le \frac14 (2h)(2b)  + \frac14 (2a-2h)(2b) = ab,
\]
hence, equality must occur in the above inequality, and thus $\abs{E_1}=hb,\abs{E_2}=(a-h)b$. Finally, we can invoke Lemma \ref{lm:43io0568} twice to deduce that $E_1,E_2$ are rectangles up to translations. This concludes the proof of the equality case.
\end{proof}
\section{Isoperimetric profiles}
In this section, we prove our main results around the isoperimetric profile functions.
\subsection{The case $\mathcal I_{G,\ell_1}(x)$}
Here we show the first main result, namely Theorem \ref{thm:main result}. To this aim, we carry out, in the next theorem, the main analysis.
\begin{theorem}\label{thm:key estimate}
Let $G$ be the rectangular lattice generated by two vectors $(a,0)$, $(0,b)$, $a,b>0$, and let us denote $D=[0,a]\times [0,b]$. Suppose that $\mathbf T$ is a $\ell_1$-isoperimetric double tiling relative to $G$ generated by $E_1, E_2$. 
Then, we have
\begin{equation}
 \Per_\phi(\mathbf T) \ge \min \ENCLOSE{a+2b, b+2a, a+b+2\sqrt{\abs{E_1}}, a+b+2\sqrt{\abs{E_2}}}.
\label{eq:isop tiling G}
\end{equation}
Finally, suppose that equality holds for some $\mathbf T$. 
Then, up to translations, one of the following occurs:
\begin{itemize}
    \item[{\rm o)}] $\mathbf T$ is degenerate (either $\abs{E_1}=0$ or $\abs{E_2}=0$) and generated by $D$;
    \item[{\rm i)}] $\mathbf T$ is generated by a square and the chipped rectangle $D$ (see Figure \ref{fig:small quadratino});
    \item[{\rm ii)}] $\mathbf T$ is generated by two adjacent rectangles whose union has closure equal to $D$ (see Figure \ref{fig:strip}).
\end{itemize}
\end{theorem}
\begin{proof}
To show \eqref{eq:isop tiling G}, we can equivalently assume that $\mathbf T$ is the $\ell_1$-isoperimetric double tiling given by Corollary \ref{cor:planar regularity} with same prescrived areas (being its perimeter the same), and in particular we can assume that each generator $E_i$ of $\mathbf T$ has an open and bounded representative satisfying $\HH^1(\partial E_i\setminus \partial^*E_i)=0$ for $i=1,2$. We subdivide the proof into multiple sub-cases, and we write $P\coloneqq \Per_\varphi(\mathbf T)$ for brevity.\medskip

\noindent\textsc{Case 1}: suppose every horizontal line and every vertical line touches $\bar E_2 + G$.
Then Proposition~\ref{prop:326576} applies and we obtain 
$P \ge a+ b + 2 \sqrt{\abs{E_1}}$. \medskip 

\noindent\textsc{Case 2}: Suppose every horizontal line and every vertical line touches $\bar E_1 +G$.
Then exchanging $E_1$ and $E_2$ we reduce to Case 1. Hence we obtain 
$P \ge a+ b + 2 \sqrt{\abs{E_2}}$.\medskip 

\noindent\textsc{Case 3}: Suppose Case 1 does not hold. We have two subcases: 3.1 and 3.2.\medskip 

\noindent\textsc{Case 3.1}: Suppose there exists a vertical line $v_1$ not touching $E_2$.
This means that $v_1\subset \bar E_1+G$.\medskip 

\noindent\textsc{Case 3.1.1}: Suppose every vertical line touches $\bar E_1+G$.
Every horizontal line touches $v_1$ (from Case 3.1) and hence touches $\bar E_1+G$.
So Proposition~\ref{prop:326576} applies with $E_1$ and $E_2$ swapped, giving $P\ge a+b+2\sqrt{\abs{E_2}}$.\medskip 

\noindent\textsc{Case 3.1.2}: Suppose 3.1.1 does not apply.
Then, there exists a vertical line not touching $\bar E_1+G$, which means 
that it is entirely contained in $\bar E_2 +G$.
Then Proposition~\ref{prop:39899} applies giving $P\ge a+2b$.\medskip 

\noindent\textsc{Case 3.2}: Suppose there exists a horizonal line not touching $E_2$.
The same subcase analysis performed in Case 3.1 (swapping $a$, $b$) gives $P\ge \min\ENCLOSE{2a+b,a+b+2\sqrt{\abs{E_2}}}$.\medskip 

\noindent\textsc{Case 4}: Suppose Case 2 does not hold. In this case we swap $E_1$ and $E_2$
and reduce to Case 3.\medskip 

\noindent\textsc{Conclusion}: We note that, combining all the cases, we have deduced the isoperimetric estimate \eqref{eq:isop tiling G}.  Therefore, to conclude the proof, we only need to characterize the equality case, and show that one of the conclusions o),i),ii) holds. We shall first characterize the equality case for the isoperimetric tiling $\mathbf T$ given by Corollary \ref{cor:planar regularity}, so that each of the previous cases applies. Then, we explain how to conclude the proof.

By symmetry, we shall suppose that $a\ge b$, and $\abs{E_1}\le \abs{E_2}$. 
Hence either $P = a+2b$ or $P = a+b+2\sqrt{\abs{E_1}}$. 
Suppose we are in Case 1: in this case, the equality conclusion in Proposition~\ref{prop:326576} yields conclusion i), if $\abs{E_1}$ is strictly positive or conclusion o) if $\abs{E_1}=0$.
Case 2 cannot happen because we have assumed $\abs{E_1}\le \abs{E_2}$ and $P\le a+ b + 2\sqrt{\abs{E_1}}$ by the minimality assumption.
Case 3.1.1 cannot happen for the same reason, so suppose that Case 3.1.2 holds. 
In this case $P=a+2b$, and the equality conclusion in Proposition~\ref{prop:39899}  yields conclusion ii).
Finally, Case 3.2 follows by the very same argument, as well as Case 4. All the cases have been analyzed; hence the tiling $\mathbf T$ given by Corollary \ref{cor:planar regularity} satisfies one of the properties o),i),ii).

Finally, if $\mathbf T$ is an arbitrary isoperimetric 2-tiling relative to $G$, then we recall that the proof of Theorem \ref{thm:existence} builds another isoperimetric 2-tiling $\mathbf{T}'$ with essentially bounded generators by possibly countable many translations (cf.\ \eqref{eq:decomposition}), and Corollary \ref{cor:planar regularity} deduces the regularity results precisely for such tiling $\mathbf T'$. As the conclusion applies to $\mathbf{ T}'$, one of the properties o),i),ii) must hold. Consequently, the original tiling $\mathbf T$ must also satisfy one of o),i),ii) by recalling again the construction of $\mathbf{ T}'$ via \eqref{eq:decomposition}.
\end{proof}
We are finally ready to show our first main result.
\begin{proof}[Proof of Theorem \ref{thm:main result}]
    The proof directly follows from Theorem \ref{thm:key estimate}. We observe that $4\sqrt{x} + 2a + 2b$ is precisely the sum of the $\ell_1$-perimeters of an axis-aligned square of area $x$ and of the chipped rectangle as in Figure \ref{fig:small quadratino}. The sum of the $\ell_1$-perimeter of two adjacent rectangles as in Figure \ref{fig:strip}  is given by $2a+4b$. In particular, if $x < \frac{b^2}{4}$ then $ 4\sqrt{x} + 2a + 2b <  2a+4b$, while the opposite holds if $\frac{b^2}{4} < x \le \frac 12$. Taking into account Theorem \ref{thm:key estimate}, we thus obtain the explicit value of $\mathcal I_{G,\ell_1}(x)$ for all $x\in[0,1]$. Finally, the conclusions o),i),ii) are in turn obtained by that of Theorem \ref{thm:key estimate}.
\end{proof}
\subsection{The case $\mathcal I_{\ell_1}(x)$}
In this part, we show our second main result when we further minimize among all possible lattices.
\begin{proof}[Proof of Theorem \ref{thm:pitagora}]
    Thanks to the Wulff isoperimetric inequality \eqref{eq:isoperimetrica quadrata}, we directly deduce 
    \[
        \mathcal I_{\ell_1}(x)\ge 2\sqrt x + 2\sqrt{1-x}.
    \]
    However, the Pythagorean double tiling (cf.\ Figure \ref{fig:pitagora}) where the two squares, with edges parallel to the coordinate axes, have respectively area $x$ and $(1-x)$ is a competitor, whence equality must occur.

    Let us now prove that in the case $x\neq \frac 1 2$, the 
    Pythagorean tilings are the only possible minimizers.
    Without loss of generality, we suppose $x< \frac 1 2$.
    Suppose that $\mathbf T$ is an $\ell_1$-isoperimetric double tiling and denote $G$ its lattice. 
    Namely, $\mathbf T$ is periodic relative to $G$ and satisfies $\Per_{\ell_1}(\mathbf T) = 2\sqrt x + 2\sqrt{1-x}$. 
    Then, by the equality case in the Wulff isoperimetric inequality \eqref{eq:isoperimetrica quadrata}, we deduce that the two generators $E_1, E_2$ are respectively equivalent to two squares with edges parallel to the coordinate axes, and matching areas. 
    Clearly each edge of $E_2$ (the largest square) touches a translation of $E_2$ itself because if not, $E_2$ should be surrounded by translations of $E_1$ which is not possible since $\abs{E_1}$ is smaller than $\abs{E_2}$ in our assumptions. 
    Denote $a=\sqrt {1-x}$ be the side length of $E_2$.
    Hence $(a,h),(w,a)\in G$ for some $h\in[0,a)$ and 
    some $w\in [0,a)$. 
    The complement of $G + E_2$ is composed by rectangles of sides $w$ and $h$, 
    but $E_1$ is also a square hence $w=h=\sqrt x$ and 
    we obtain a Pythagorean double tiling.

    The remaining case is when $x=\frac 1 2$.
    In this case $E_1$ and $E_2$ are squares with 
    the same side length $a=\frac{1}{\sqrt 2}$.
    It is easy to conclude that the squares,
    having the edges aligned with the coordinate axes,
    must fill the plane in horizontal or vertical stripes,
    as claimed in the statement.
\end{proof}
\section{Locally $\ell_1$-isoperimetric partitions}
In this section, we prove the main result Theorem \ref{thm:locally isop}. We first give the definition of a locally $\phi$-isoperimetric partition for a general anisotropic norm $\phi$. 
\begin{definition}[locally $\phi$-isoperimetric partition]\label{def:locally anisotropic isop}
    Let $\phi$ be a norm in $\R^d$, and let $(E_n)$ be a Borel partition (i.e.\ $|E_i\cap E_i|=0$ for $i\neq j$, and $|\R^d\setminus \cup_n E_n|=0$). We say that $(E_n)$ is locally $\phi$-isoperimetric if for all $R>0$ and for all Borel partitions $(F_n)$ satisfying 
    \[
        |E_n \cap Q_R|=|F_n\cap Q_R|,\qquad E_n\triangle F_n \Subset Q_R,
    \]
    where we denoted $Q_R\coloneqq (-R,R)^d$, it holds that
    \[
        \sum_{n\in\N} \Per_\phi(E_n,Q_R)\le \sum_{n\in\N}\Per_\phi(F_n,Q_R).
    \]
\end{definition}
Next, we show a useful local minimality property, and we lastly prove our main result.
\begin{lemma}\label{lm:51fdfg}
Let $G$ be a lattice and let $\phi$ be a norm in $\R^d$. 
Let $\mathbf T$ be a $\phi$-isoperimetric $2$-tiling relative to $G$.
Let $\Omega\subset \R^2$ be an open set such that $\bar \Omega \cap (\overline\Omega +g)=\emptyset$ for every $g\in G\setminus\{0\}$.
Suppose that there exists $E_1\in \mathbf T$ such that $E_1\Subset \Omega$
and let $E_2$ be a second generator.
Let $g_n$, $n\in \N$ be an enumeration of $G$
Then the partition 
\[
   (E_1, E_2 + g_1, \dots, E_2 + g_n, \dots)
\]
is a volume constrained minimizer of the $\phi$-perimeter relative to $\Omega$ (cf.\ Definition \ref{def:local minimal}).
\end{lemma}
\begin{proof}
Consider a fundamental domain $D\supset \Omega$,\ i.e.\ $\R^d = \bigcup_g (D+g)$, $(D+g)\cap D =\emptyset$
for all $g\in G\setminus\{0\}$.
Suppose also that $\Per_\phi (E_1,\partial D) = \Per_\phi(E_2,\partial D) = 0$. Define 
\[
   \tilde E_1 = F, \qquad 
   \tilde E_2 = \bigcup_n (F_n\cap (D - g_n)),
   \qquad \tilde {\mathbf T} = \{\tilde  E_1+g\colon g\in G\} \cup \{ \tilde E_2+g\colon g\in G\}
\]
so that $\tilde {\mathbf T}$ is a 2-tiling relative to $G$.
We observe that since $|\tilde E_1| = |E_1|$ then
\[
    |\tilde E_2| = |D \setminus \tilde E_1| = | D \setminus E_1 | = |E_2|.
\]
Using that $\mathbf T$ is $\phi$-isoperimetric relative to $G$ and by similar arguments as already faced in the proof of Lemma \ref{lem:local minimality}, we get
\begin{align*}
    0&\le  2\Per_\phi(\mathbf{\tilde T}) - 2\Per_\phi(\mathbf T)=  \Per_\phi(F)- \Per_\phi(E_1) + \sum_n \Per_\phi(F_n,\Omega) -  \sum_n \Per_\phi(E_2+g_n,\Omega),
\end{align*}
which is the desired conclusion.
\end{proof}
\begin{proof}[Proof of Theorem \ref{thm:locally isop}]
Let $R>0$ be given such that $E_1 \Subset Q_R=(-R,R)^2$.
Consider the group $G=2R\cdot \Z^2$ generated by $g_1 = (2R,0)$, $g_2=(0,2R)$
and let $\mathbf T$ be a double tiling relative to $G$ 
generated by $E_1$ and $\tilde E_2 := [0,2R]^2 \setminus (E_1 + G)$.
By Theorem \ref{thm:main result}, we know that $\mathbf T$ is $\ell_1$-isoperimetric relative to $G$ up to choosing $R$ big enough.
By Lemma~\ref{lm:51fdfg} we also know that the partition $\mathbf E$ is a volume constrained minimizer of the $\ell_1$-perimeter relative to $Q_R$. Since $R>0$ was arbitrary, we conclude that the partition $\mathbf E$ is locally $\ell_1$-isoperimetric according to Definition \ref{def:locally anisotropic isop}.

The second part of the statement (about the \emph{vortex} partition) can be obtained by the 
same argument invoking instead Theorem~\ref{thm:pitagora}.
\end{proof}

\medskip

\noindent\textbf{Acknowledgments}. 
All authors are members of INDAM-GNAMPA and acknowledge support from the MIUR Excellence Department Project awarded to the Department of Mathematics, University of Pisa, CUP I57G22000700001.

F.N acknowledges support by the European Union (ERC ConFine, 101078057).

M.N. acknowledges partial support from the PRIN 2022 project 2022E9CF89, and the PRIN 2022 PNRR project P2022WJW9H.

E.P. acknowledges support from the PRIN 2022 project 2022PJ9EFL, and from the project PRA 2022 14 GeoDom (Università di Pisa).

F.N. and E.P. acknowledge support from the INdAM-GNAMPA Project ``Analisi e Gamma-convergenza per alcuni funzionali non locali'' CUP E53C25002010001\#.

\newcommand{\etalchar}[1]{$^{#1}$}

\end{document}